\documentclass[11pt]{article}

\usepackage[margin=1in]{geometry} %page settings
\usepackage{graphicx,caption}
\graphicspath{ {./figures/} }

\usepackage{amsmath, mathtools, thmtools, amsthm}
\usepackage{mathrsfs}
\usepackage{enumerate} 
\usepackage{amsfonts,bm, braket}
\usepackage[scr=boondox]{mathalfa}
\usepackage{xcolor}

\usepackage{bookmark}
\usepackage{setspace}
\usepackage[backend=bibtex,style= numeric]{biblatex}

\usepackage[utf8]{inputenc}
\newtheorem{theorem}{Theorem}[section]

\newtheorem{lemma}[theorem]{Lemma}

\theoremstyle{definition}

\theoremstyle{definition}
\newtheorem*{defn}{Definition}
\theoremstyle{remark}
\newtheorem*{rem}{Remark}
\theoremstyle{empty}

\newcommand{\br}[1]{\left( #1 \right)}
\newcommand{\textbi}[1] {\textbf{\textit{#1}}}
\newcommand{\comment}[1]{}
\DeclarePairedDelimiter\abs{\lvert}{\rvert}%
\DeclarePairedDelimiter\ceil{\lceil}{\rceil}%

\def\reals{{\mathbb R}}

\def \R {{\mathscr{R}}}
\def\S{{\cal S}}
\def\P{{\cal P}}

\def \F {{\mathscr{F}}}

\def \H {{\mathscr{H}}}
\def \G {{\mathscr{G}}}

\def \conv {{\sf{conv}}}

\title{Helly-Type Theorems for Splitting Point Sets \thanks{A preliminary version of this article appeared in the Proceedings of the 37th ACM-SIAM Syposium on Discrete Algorithms (SODA), January 11-14,2026, Vancouver, Canada.}}
\author{ Lidor Portal \thanks{ Department of Computer Science, Ben-Gurion University of the Negev, Beer-Sheva, Israel 84105, \texttt{lidorpor@post.bgu.ac.il}. Supported by the Lynn and William Frankel Center for Computer Science at Ben-Gurion University and by grants 2891/21 and 3010/25 from Israel Science Foundation.} \and Natan Rubin \thanks{Department of Computer Science, Ben-Gurion University of the Negev, Beer-Sheva, Israel 84105, \texttt{rubinnat.ac@gmail.com}. Supported by grants 2891/21 and 3010/25 from Israel Science Foundation.}}

\begin{document}
\maketitle

\begin{abstract}
Let $0 < \alpha \leq 1/2$. We say that a finite point set $P$ in $\reals^d$ is {\it $\alpha$-split} by a hyperplane $h$ if each of the closed half-spaces determined by $h$, contains at least $\alpha \abs{P}$ of the points of
$P$. We further say $P$ is $\alpha$-split by a $k$-dimensional flat $\tau$ if $P$ is $\alpha$-split by {\it any} hyperplane through $\tau$.  In the standard notation (which coincides with Tukey depth for $k= 0$), the $k$-flat $\tau$ has depth $\alpha$ with respect to $P$.

We establish interesting Helly-type theorems for splitting families of finite point sets in $\reals^d$. Unlike the classical sufficient Helly-type criteria for transversals to families of compact convex sets,
which exist only for point and hyperplanes, our results extend to splitting families of point sets by collections of $k$-flats of arbitrary dimensionality $ 0 \leq k \leq d-1$.

\end{abstract}

\newpage
\section{Introduction}

A {\it $k$-flat} in $\reals^d$ is a translate of a linear vector subspace of dimension $k$;  Flats are also known as {\it affine} subspaces \cite[Ch. 1.1]{matousek2013lectures}.
In particular, $0$-flats are points in space and $(d-1)$-flats are {\it hyperplanes}. 

The Ham-Sandwich Theorem \cite{st1942ham} implies that any $d$ finite point sets $P_1,\ldots,P_d$ in $\reals^d$ can be simultaneously equipartitioned by a hyperplane $H$, so that each of the induced open halfspaces in $\reals^d\setminus H$ contains at most $|P_i|/2$ of the points of each set $P_i$, for $1\leq i\leq d$. More generally, the Central Transversal Theorem \cite{Dol92,ZV90} implies that any $k+1$ finite point sets can be partitioned by a $k$-flat $\tau$, such that any closed half-space containing $\tau$ also contains at least $\frac{|P_i|}{d-k+1}$ points of $P_i$, for all $1\leq i \leq k+1$.   
In the past decades, this statement of distinctly topological flavour \cite{st1942ham} had been extended in manifold and exciting ways to partitions of families of point sets and continuous distributions by a small number of hyperplanes \cite{roldan2022survey}. 
Many of these results bear strong relation to the study of Helly-type theorems \cite{eckhoff1993helly}.

Unfortunately, hardly any of these theorems apply to scenarios where the number of point sets is far larger than the dimension $d$ of the underlying Euclidean space. For example, consider a family $\mathcal{P} = \{P_1,\dotsc, P_n\}$ comprised of $n\gg d$ point sets which are contained in infinitesimally small balls whose centers are in general position. Since no hyperplane crosses the convex hulls of more than $d$ member sets of $\mathcal{P}$, it follows that no more than $d$ of them can be split (for any proportion) by a single hyperplane(see Figure \ref{fig:no-ham-sandwich} for visual illustration).

\begin{figure}[h]
  \centering
  \includegraphics[scale = 0.4]{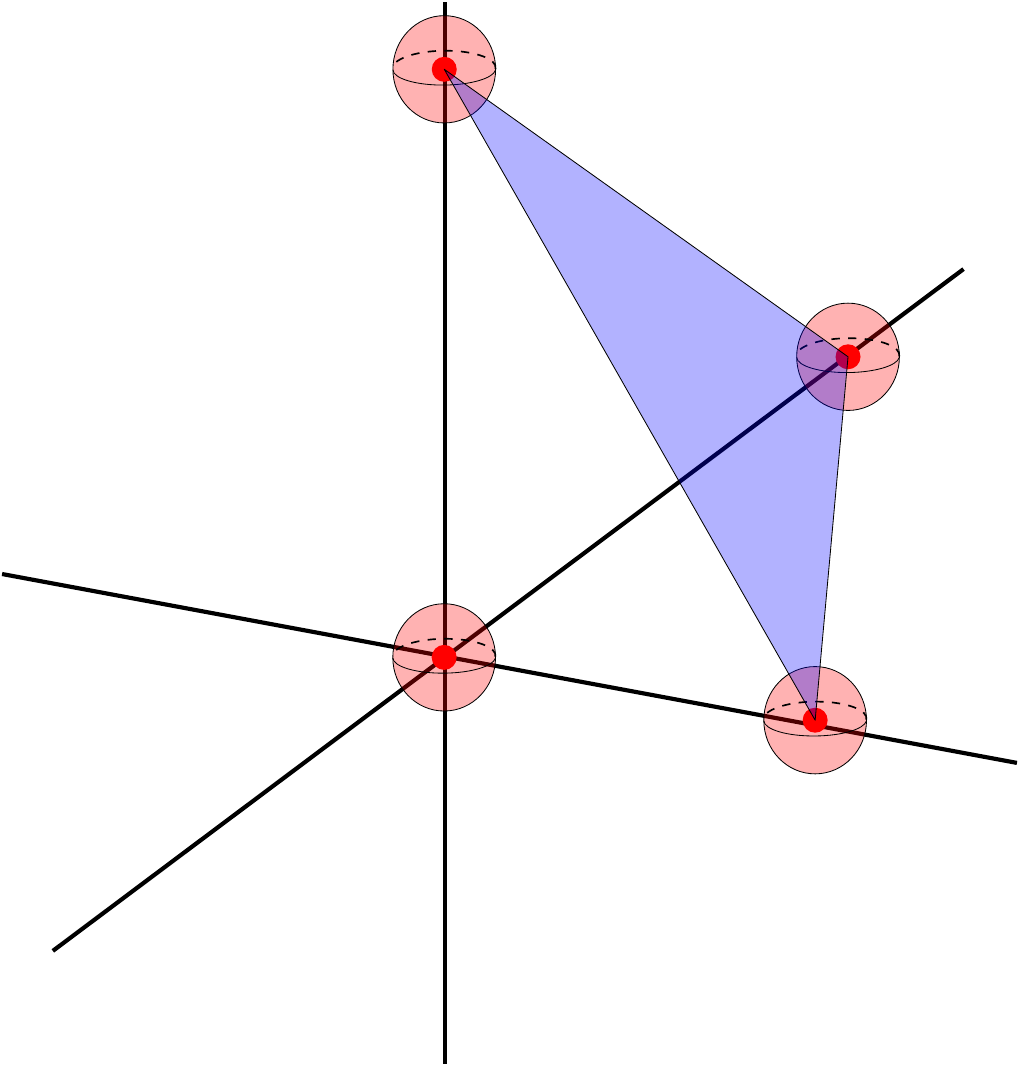}
  \caption{\fontfamily{cmss}\selectfont \footnotesize
  No (hyper)plane in $\reals^3$ crosses the centers of all 4 sets. Thus, for any proportion $\alpha\in(0,1/2]$ no (hyper)plane $\alpha$-splits all 4 sets simultaneously.}
  \label{fig:no-ham-sandwich}
\end{figure}

In this paper, we establish interesting {\it Helly-type} conditions for splitting arbitrary finite families $\{P_1,\ldots,P_n\}$ of finite $d$-dimensional point sets.

\noindent {\bf Helly-type theorems.} Let $k<d$ be non-negative integers. We say that a collection $\H$ of $k$-flats is a {\it transversal} to a family $\F$ of sets in $\reals^d$ if every set in $\F$ is intersected by at least one element of $\H$.

The celebrated theorem of Helly \cite{helly1923mengen} states that any family $\F$ of at least $d+1$ compact convex sets in $\reals^d$ has a non-empty intersection $\bigcap \F\neq \emptyset$ if and only if any $(d+1)$-size subset $\G\in {\F\choose d+1}$ does. In other words, such a family $\F$ admits a transversal of a single point (i.e., a $0$-flat).
 
In the past 50 years Geometric Transversal Theory has been preoccupied with the following questions (see, e.g, \cite{Amenta2015HellysTN,danzer1963helly,eckhoff1993helly,goodman1993new}):

\begin{itemize}
  \item Does Helly’s Theorem generalize to transversals by $k$-flats, for $1 \leq k \leq d-1$?
  \item Given that a significant fraction of the $(d+1)$-tuples $\G \in  \binom{\F}{d+1}$ have a non-empty intersection, can the family $\F$ of convex sets, or at least some fixed fraction of its members, be pierced by constantly many points?
\end{itemize}

The first question has been settled to the negative already for $k=1$. For instance, Santal\'o \cite{santalo1940teorema} and Danzer \cite{danzer1957problem} observed that for any $n \geq 3$ there are families $\F$ of $n$ convex sets in $\reals^2$ so that any $n-1$ of the sets can be crossed by a single line transversal while no such transversal exists for $\F$. Nevertheless, Alon and Kalai \cite{alon1995bounding} showed that the following almost-Helly property holds for $k = d-1$: If every $d+1$ (or fewer) of the sets of $\F$ can be crossed by a hyperplane, then $\F$ admits a transversal by $C$ hyperplanes, where the number
$C= C(d)$ depends only on the dimension $d$. While the properties of hyperplane transversals largely resemble those of point transversals, this is not the case for transversals by $k$-flats of intermediate dimensions $1 \leq k \leq d-2$. This phenomenon can be largely attributed to the complex topological structure of the space of transversal $k$-flats.

The second question has given rise to such fundemental results as the so called fractional Helly Theorem \cite{katchalski1979problem}, Lov{\'a}sz's Colored Helly Theorem \cite{BARANY1982141} and Alon and Kleitman's $(p,q)$-theorem \cite{alon1992piercing}; see \cite{eckhoff1993helly,barany2022,wenger2004helly} for a more comprehensive exposition.
The first result in this series states that if {\it sufficiently many} of the $(d+1)$-tuples $\G\subseteq \F$ have non-empty intersections $\bigcap \G$ then some fraction of the sets can be simultaneously intersected by a single point; See Section \ref{Sec:Prelim} for a more comprehensive discussion. 

\noindent {\bf The $\bm{(p,q)}$-theorems for transversals to convex sets.}
A seminal result of Alon and Kleitman \cite{alon1992piercing} yields a small-size point transversal if the Helly's condition on the $(d+1)$-tuples of convex sets is replaced by a certain {\it $(p,q)$-property}, which we state in a more general form.

\noindent {\bf Definition.} For any positive integers $p\geq q$, 
a family $\F$ of sets is said to have the {\it $(p,q)$-property} if among any $p$ members of $\F$, some $q$ have a non-empty common intersection. 

A family $\F$ of {\it compact convex sets} is said to have the {\it $(p,q)$-property with respect to $k$-flats} (i.e., translates of $k$-dimensional linear subspaces \cite{matousek2013lectures}) if among any $p$ members of $\F$, some $q$ can be simultaneously crossed by a single $k$-flat. 

\begin{theorem}[Alon and Kleitman \cite{alon1992piercing}]\label{Theorem:AlonKleitman}
For any positive integers $p,q$ and $d$ that satisfy $p\geq q\geq d+1$, there exists a finite number $C(p,q,d)$ such that the following statement holds: any finite family $\F$ of  convex sets in $\reals^d$ that satisfies the $(p,q)$-property with respect to points, admits a transversal by $C$ points.
\end{theorem}

\noindent See \cite{KST18improved} for improved bounds on the constant $C=C(p,q,d)$ in Theorem \ref{Theorem:AlonKleitman}.  

A parallel $(p,q)$-theorem for {\it hyperplane transversals} was established by Alon and Kalai.

\begin{theorem}[Alon and Kalai \cite{alon1995bounding}]\label{Theorem:AlonKalai}
For any positive integers $p,q$ and $d$ that satisfy $p\geq q\geq d+1$, there exists a finite number $C(p,q,d)$ such that the following statement holds: any family $\F$ of convex sets in $
\reals^d$ that satisfies the $(p,q)$-property with respect to hyperplanes admits a transversal by $C$ hyperplanes.
\end{theorem}

Notice that the statement of Theorem \ref{Theorem:AlonKleitman} does not hold for $q \leq d$ : any collection $\F$ of $n$
hyperplanes in general position in $\reals^d$ has the $(d,d)$-property, yet does not allow a transversal by fewer than $\lceil\frac{n}{d}\rceil$ points. To see that the assumption $q  \geq d+1$ is also essential for Theorem \ref{Theorem:AlonKalai}, it is enough to consider a family of $n$ infinitesimally small balls in $\reals^d$ so that no $d+1$ of them lie on the same hyperplane; see Figure \ref{fig:no-ham-sandwich}.

As was demonstrated by Alon, Kalai, Matou\v{s}ek, and Meshulam \cite{alon2002transversal}, no comparable $(p,q)$-theorems exist for $k$-flats of intermediate dimensionality $1\leq k\leq d-2$: for any choice of $d\geq 3$, $1\leq k\leq d-2$, and an integer $h>k$, one can construct arbitrary large families of convex sets so that any $h$ of their members can be crossed by a single $k$-flat, yet no $h+4$ among them can.

\noindent {\bf $\mathbf{\alpha}$-splitting point sets with flats.} 
Let $\alpha \in (0,\frac{1}{2}]$ and $P$ be a finite point set in $\reals^d$.
We say that $P$ is $\alpha$-{\it split} by a hyperplane $h$ if each of the closed half-spaces that are supported by $h$ contains at least $\alpha |P|$ points of $P$.

We further say that $P$ is $\alpha$-{\it split} by a $k$-flat $\tau$, for $0\leq k\leq d-1$, if it is $\alpha$-\textit{split} by every hyperplane through $\tau$ (see Figure \ref{fig:alpha_splitting_flat}). 

\begin{figure}[h]
	\centering
	\includegraphics[scale = 0.5]{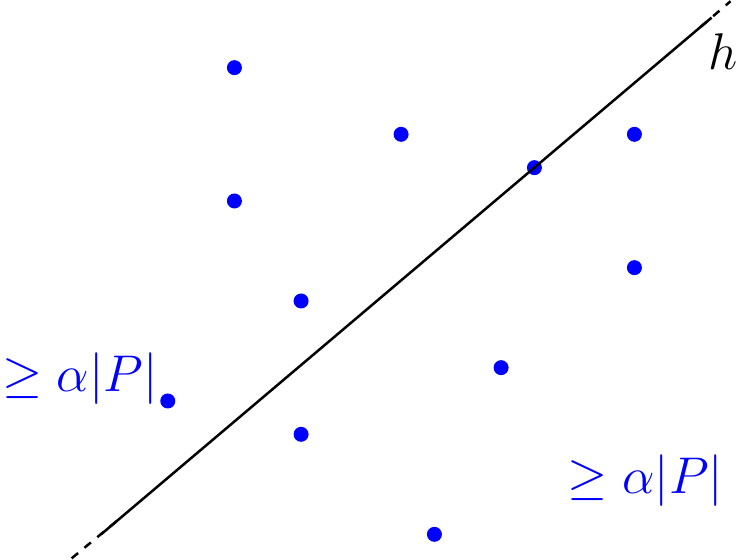}
	\caption {\fontfamily{cmss} \footnotesize $\alpha$-splitting property with respect to hyperplanes. Each closed half-space determined by $h$ contains at least an $\alpha$ fraction of the points of $P$. }
	\label{fig:alpha_splitting_flat}
\end{figure}

In the notation of Bukh, Matou\v{s}ek and Nivasch \cite{bukh2010stabbing}, such $k$-flats are said to have {\it depth $\alpha$}. 

\noindent A point $p\in \reals^d$ is said to have {\it depth} $\alpha$ with respect to $P$ if every closed half-space containing $p$ contains at least $\alpha|P|$ of the points of $P$. \\ 
More generally, let $\tau$ be a $k$-flat ($0\leq k \leq d-1$). The {\it depth} of $\tau$ as defined in \cite{magazinov2018improvement} is 
\begin{equation*}\label{eq:depth}
	\text{depth}_P(\tau)= \inf\left\{|H\cap P| \mid H \text{ is a closed half-space}, \tau \subset H\right\}
\end{equation*}

According to Rado's centerpoint theorem \cite{rado1946theorem}(also see Theorem \ref{Theorem:Centerpoint} in Section \ref{Sec:Prelim}), any finite point set in $\reals^d$ is $\left(\frac{1}{d+1}\right)$-split by at least one point. Furthermore, it is conjectured that every point set $P$ is $\left(\frac{k+1}{k+d+1}\right)$-split by at least one $k$-flat, for $0\leq k\leq d-1$ \cite{bukh2010stabbing}; recent partial results in this direction have been reported by Magazinov and P\'or \cite{magazinov2018improvement}.

\noindent {\bf The $\bm{(p,q,\alpha)}$-property for splitting point sets.}  For $p\geq q\geq d+1$, we say that a family $\P=\set{P_1,\dotsc,P_n}$ of finite point sets  has the {\it $(p,q,\alpha)$-property} with respect to \textit{$k$-flats} in $\reals^{d}$ if for any collection of $p$ members of $\P$, some $q$ among them can be simultaneously $\alpha$-split by the same $k$-flat $\tau$.  

We say that such a family $\P=\{P_1,\ldots,P_n\}$ is {\it $\alpha$-split} by a collection $Q$ of $k$-flats if each set $P_i\in \P$ is $\alpha$-split by at least one element of $Q$.

It is well-known that the set of all $\alpha$-deep points with respect to a finite point set $P$ gives rise to a (possibly empty) convex polytope $S_\alpha(P)$\footnote{Namely, $S_\alpha (P)$ is the intersection of all the open halfspaces that encompass more than $(1-\alpha)\abs{P}$ points of $P$. As is easy to check (see e.g. \cite[Ch. 1.4]{matousek2013lectures}), this set is either empty or a convex polytope.}. Consequently, the $(p,q,\alpha)$-property of $\P=\{P_1,\ldots,P_n\}$  \textbf{\textit{with respect to points}} is equivalent to the $(p,q)$-property of their respective ``$\alpha$-centersets'' $S_\alpha(P_i)$. Hence, applying the $(p,q)$-theorem of Alon and Kleitman (Theorem \ref{Theorem:AlonKleitman}) to the resulting family $\{S_\alpha(P_i)\mid 1\leq i\leq n\}$ yields a set $Q$ of $C(p,q,d)$ points which includes at least one $\alpha$-center for each member set $P_i\in \P$. In other words, we have established a ``$(p,q)$-theorem'' for $\alpha$-centerpoints. Unfortunately, this ``lossless'' argument does not extend to $\alpha$-splitting point sets with hyperplanes or $k$-flats of strictly positive dimensionality $k$: an $\alpha$-splitting hyperplane for a point set $P$ may nevertheless miss its centerset $S_\alpha(P)$ as illustrated in the following example (see Figure \ref{fig:splitting_hyperplane_misses_center}).

\begin{figure}[h]
  \centering
  \includegraphics[scale = 0.7]{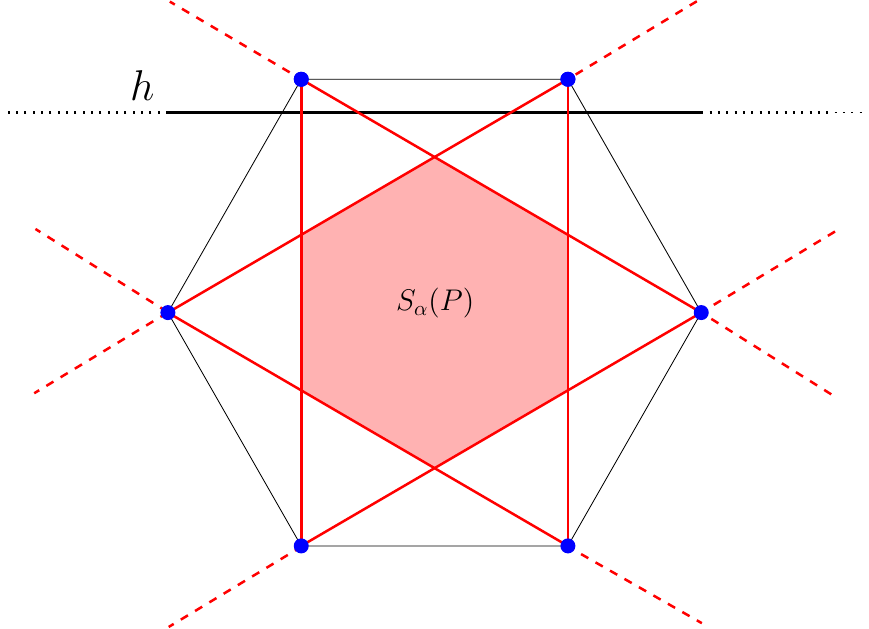}
  \caption{\fontfamily{cmss}\selectfont \footnotesize 
  For $\alpha = \frac{1}{d+1}$ and the set of blue points $P$ in $\reals^2$, the $\alpha$-centerset of $P$ is the set of all centerpoints of $P$ which is never empty by Rado's centerpoint theorem (Theorem \ref{Theorem:Centerpoint}). The set of all centerpoints is the intersection of all open half-spaces that contain more than $(1-\alpha)\abs{P} = \frac{d}{d+1}\abs{P}=\frac{2}{3}\abs{P} = 4$ points of $P$ (the area shaded in red in the figure), and thus is a convex polytope. However, it is evident that $h$ is $\alpha$-splitting $P$ despite not crossing $S_\alpha(P)$.}
  \label{fig:splitting_hyperplane_misses_center}
\end{figure}

\noindent {\bf Our results.} Our first contribution is the following $(p,q)$-type statement for $\alpha$-splitting finite point sets with hyperplanes. 

\begin{theorem}\label{Theorem:SplitHyperplanes}

For any $\alpha \in (0,\frac{1}{2}]$, and any integers $p\geq q \geq d+1$ there exist constants $C=C(p,q,d)< \infty$ and $
\beta=\beta(\alpha,d)\in (0,\alpha)$ such that the following statement holds:
  
Every family $\P = \set{P_1,\dots, P_n}$ of $n\geq p$ finite point sets in $\reals^d$ that has the $(p,q,\alpha)$-property with respect to hyperplanes, can be $\beta$-split by a set $Q$ of at most $C$ hyperplanes. 
  
\end{theorem}

Notice that any hyperplane (or, more generally, $k$-flat) that $\alpha$-splits a given sub-family $\P'\subseteq \P$ must, in particular, cross the convex hulls $\conv(P_i)$ of the member sets $P_i\in \P'$.
Hence, it hardly comes as a surprise that Theorem \ref{Theorem:SplitHyperplanes} can be deduced from Theorem \ref{Theorem:AlonKalai}.

\noindent{\bf A substitute to the $\alpha$-centerset for $k$-flat transversals.} 
To this end, we use the simplicial partition of Matou\v{s}ek \cite{matousek1992reporting} to 
show that any point set $P$ determines a more robust analogue $\tilde{S}_\alpha(P)$ of the $\alpha$-centerset $S_\alpha(P)$, with the following properties:

\begin{itemize}
\item any $\alpha$-splitting $k$-flat of $P$ must cross $\tilde{S}_\alpha(P)$ and, conversely, 
\item $P_i$ is $\beta$-split by any $k$-flat that crosses $\tilde{S}_\alpha(P)$, with a suitable constant ${\beta=\beta(\alpha,d) \in (0,\alpha)}$. 
\end{itemize}

\noindent {\bf A $\bm{(p,q)}$-type theorem for $k$-flats.} 
While this property of independent interest holds for $k$-flats of arbitrary dimensionality $0 \leq k \leq d-1$, the proposed proof of Theorem \ref{Theorem:SplitHyperplanes} does not extend to $k$-flats in the intermediate range $1\leq k \leq d-2$, where no analogue of Theorems \ref{Theorem:AlonKleitman} and \ref{Theorem:AlonKalai} can possibly exist \cite{alon2002transversal}.

Nevertheless, by combining $\epsilon$-approximations \cite{HW87} with Matou\v{s}ek’s Helly-type results for families of sets with bounded VC-dimension \cite{matousek2004bounded}, one can establish a stronger and more general $(p,q)$-type theorem. 
 
\begin{theorem}\label{Theorem:SplitFlats}
  
For every integers $d,\;0\leq k\leq d-1$ and every $p\geq q\geq (k+1)(d-k)+1$, and every reals $\alpha \in (0,\frac{1}{2}]$ and $\epsilon\in (0,\alpha)$, there is an integer $C=C(p,q,d,k,\epsilon)<\infty$ with the following property.
  
Every collection $\P = \set{P_1,\dotsc, P_n}$ of $n\geq p$ finite point-sets in $\reals^d$ that has the $(p,q,\alpha)$-property with respect to $k$-flats, can be $(\alpha-\epsilon)$-split by a family $Q$ of at most $C$ $k$-flats. 

\end{theorem}

For $k=d-1$, Theorem \ref{Theorem:SplitFlats} yields a stronger form of Theorem \ref{Theorem:SplitHyperplanes} with arbitrary good $\beta$-splitting guarantee $\beta=\alpha-
\epsilon$. (However, the cardinality of the $\beta$-splitting family $Q$ may increase as $\epsilon$ approaches $0$.)
To establish Theorem \ref{Theorem:SplitFlats}, we pass to a higher dimensional space $\reals^D$, with $D=(k+1)(d-k)$, and use the $\epsilon$-approximation theorem of Haussler and Welzl \cite{HW87} to construct for each set $P_i$ a dual region $\Lambda_{\alpha,\epsilon,k}(P_i)$ which describes an ``$\epsilon$-rounded'' superset of $\alpha$-splitting $k$-flats. The crucial observation is that the regions $\Lambda_{\alpha,\epsilon,k}(P_i)$ have a bounded Vapnik-Chervonenkis dimension (i.e., VC-dimension) \cite{vc1971}, and their dual shatter function satisfies $\pi(m)=o(m^D)$.
 To deduce a $(p,q)$-theorem for the regions $\Lambda_{\alpha,\epsilon,k}(P_i)$ of $(\alpha-\epsilon)$-splitting $k$-flats, we will use the connection between the VC-dimension and the so called fractional Helly-numbers, that was observed by Matou\v{s}ek \cite{matousek2004bounded}. 

\medskip
\noindent {\bf Related work.} As was previously mentioned, there are remarkably few results on splitting {\it large} collections $\P=\{P_1,\ldots,P_n\}$ of finite point sets with hyperplanes, let alone $k$-flats.  
The so called Polynomial Ham-Sandwich Theorem of Stone and Tukey \cite{st1942ham} states that any such family $\P$ can be equipartitioned by a hypersurface of degree $O(n^{1/d})$.
A somewhat related SODA 2023 result of Har-Peled and Zheng \cite{harpeled2022} yields an efficient approximation algorithm for the {\it optimisation problem} of finding the smallest possible set $H$ of hyperplanes whose arrangement $\reals^d\setminus\left(\bigcup H\right)$ simultaneously subdivides all sets $P_i\in \P$ into parts that encompass  {\it at most} a prescribed fraction $\gamma_i$ of each set $P_i$, with $\gamma_i \in (0,1)$. (However, it is possible that no set $P_i$ is $(1-\gamma_i)$-split by even a single hyperplane $h$ in the target set $H$.)

\bigskip
\noindent {\bf Paper organization.} The rest of this paper is organized as follows. In Section \ref{Sec:Prelim} we introduce the fundamental geometric notions and facts that underlie the proofs of Theorems \ref{Theorem:SplitHyperplanes} and \ref{Theorem:SplitFlats}. These include geometric range spaces and VC-dimension \cite{vc1971,matousek2013lectures,matousek1999geometric}, $\epsilon$-nets, $\epsilon$-approximations \cite{HW87} and fractional Helly numbers. We also present the fundamental relation between fractional Helly numbers, $(p,q)$-type theorems, and dual shatter functions of set systems with bounded VC-dimension, which was observed by Matou\v{s}ek \cite{matousek2004bounded}.
In Section \ref{Sec:Hyperplanes}, we use Matou\v{s}ek's simplicial partitions \cite{matousek1992reporting} to define the aforementioned sets $\tilde{S}_\alpha(P)$, and then derive Theorem \ref{Theorem:SplitHyperplanes} from Theorem \ref{Theorem:AlonKalai}.
In Section \ref{Sec:Flats}, we use the previously mentioned connection between fractional Helly numbers and the VC-dimension \cite{matousek2004bounded} in order to establish Theorem \ref{Theorem:SplitFlats}. In Section \ref{Sec:Conclude}, we conclude with closing remarks and open problems. 

\section{Preliminaries}\label{Sec:Prelim}

To facilitate the proofs of Theorems \ref{Theorem:SplitHyperplanes} and \ref{Theorem:SplitFlats}, let us list several fundamental facts from discrete geometry. 

\begin{defn}
\label{Def:centerpoint}

Let $X$ be a finite set of points in $\reals^d$. A point $q\in \reals^d$ (not necessarily in $X$) is called a \textit{\textbf{centerpoint}} for $X$ if every closed half-space containing $q$ contains at least $\abs{X}/(d + 1)$ points of $X$.
  
\end{defn}

\begin{theorem}[Rado's Center Point Theorem \cite{rado1946theorem}]
  \label{Theorem:Centerpoint}
  For every finite point set $X$ in $\reals^d$, there exists a centerpoint.
\end{theorem}

\noindent{\bf Helly and fractional Helly numbers.}
The {\it Helly number} of a family $\F$ of sets is the smallest possible integer $k$ for which the following statement holds:
Any finite subset $\G\subseteq \F$ so that every $k$ members of $\G$ intersect, must have a non-empty intersection $\bigcap \G$. If no such integer $k$ exists, we say that the Helly number of $\F$ is $\infty$.

According to Helly's Theorem \cite{helly1923mengen}, the Helly number of the family $\F_d$ of all convex sets in $\reals^d$ is $d+1$.

The {\it fractional Helly number} of a family $\F$ of sets if the smallest possible integer $k$ for which the following statement holds: for every $\alpha>0$ there exists a number $\beta>0$ with the following property: 
\begin{itemize}
  \item[] For every finite subfamily $\G=\{F_1,\dotsc, F_n\}\subseteq \F$, if at least $\alpha \binom{n}{k}$ of the $k$-tuples $\H\in {\G\choose k}$ have non-empty intersection $\bigcap \H$, then there exists an element common to at least $\beta n$ sets $F_i$. 
\end{itemize}

The fractional Helly number is set to $\infty$ if no integer $k$ meets the aforementioned criteria.
According to the 1979 theorem of Katchalski and Liu \cite{katchalski1979problem}, the fractional Helly number of the family $\F_d$ is also $d+1$; its subsequent refinement by Kalai \cite{kalai1984intersection} yields the fractional guarantee $\beta=\beta(\alpha,d)$ that approaches $1$ as $\alpha$ tends to $1$.

\medskip
\noindent {\bf Geometric range spaces.} A {\textbf{\textit range space}} $\R$ is a pair $(X,\F)$, where $X$ is a (possibly infinite) set of geometric objects (e.g., points or hyperplanes) and $\F$ is a collection of subsets of $X$. 
Any finite subset $A\subseteq X$ yields a hypergraph $(A,\F_A)$, where $\F_A$ denotes the restriction $\{f\cap A \mid f\in \F\}$ of $\F$ to $A$.

\noindent {\bf VC-dimension.} We say a finite subset $A \subseteq X$ is {\it shattered} if the restriction $\F_A$ realizes all the possible subsets of $A$, so that $\left|\F_A\right|= 2^{|A|}$. 
    The {\it VC-dimension} of the range space $\R=(X,\F)$ is the maximum size of a subset $A \subset X$ that is shattered by $\F$ \cite[Ch. 10.2]{matousek2013lectures}. If no such maximum integer exists, we say $\R$ has {\it infinite} VC-dimension. Formally,  
        \[\text{VC-dim}\left(\R\right) = \sup \set{ \abs{A} : A \subseteq X \text{ is shattered by } \F}.\]

\noindent {\bf Shatter function.} Any range space $\R=(X,\F)$ determines {\it the shatter function} $\pi_\R: \mathbb{N} \to \mathbb{N}$ which is given by $\pi_\R(m)=\max_{A\subseteq X, \abs{A}=m}\left|\F_A\right|$. According to the fundamental lemma of Sauer and Shelah \cite{Sauer72}, any range space $\R$ of bounded VC-dimension $\text{VC-dim}(\R)=D<\infty$ yields polynomial shatter function $\pi_\R(m)=O\left(m^D\right)$.

In addition to the (primal) shatter function $\pi_\R$, any range space $\R=(X,\F)$ determines the {\it dual shatter function} $\pi_\R^\ast: \mathbb{N} \to \mathbb{N}$, so that $\pi_\R^\ast(m)$ is the maximum number of nonempty fields of the Venn diagram of $m$ sets of $\F$. More formally, two points $x,y \in X$ are called {\it equivalent} with respect to sets $F_1, \dots, F_m$ if $\{ i \in [m] : x \in F_i \} = {\{i \in [m] : y \in F_i \}}$. 
	Then $\pi_\R^\ast(m)$ is the maximum possible number of such equivalence classes over
	all choices of $F_1, \dots, F_m \in \F$. It is well known that the shatter function is polynomial if and only if the dual shatter function is.

\noindent {\bf $\bm{\epsilon}$-nets and $\bm{\epsilon}$-approximations.} Let $(A,E)$ be a hypergraph (so that $E\subseteq 2^A$) and $0<\epsilon\leq 1$. A subset $N \subseteq A$ is called an $\epsilon$-net for $(A,E)$ if it intersects every edge $F\in E$ with $|F|\geq \epsilon |A|$.

A subset $N \subseteq X$ is called an {\it $\epsilon$-approximation} for $(A,E)$ if for every $F\in \F$, we have that 
\[
\left|\frac{|F\cap A|}{|A|}-\frac{|F\cap N|}{|N|}\right|<\epsilon.
\]

\begin{theorem}[$\epsilon$-net Theorem \cite{HW87}]
\label{Theorem:Eps_net}
  There is a constant $c$ such that for any range space $\R=(X,\F)$ of finite VC-dimension $D$, any $0 < \epsilon\leq 1$, and any subset $A\subseteq X$, there exists an $\epsilon$-net for $(A,\F_A)$ whose cardinality is at most $c\cdot  \frac{D}{\epsilon}\ln\frac{D}{\epsilon}$.
\end{theorem}

\begin{theorem}[$\epsilon$-approximation Theorem \cite{vc1971,HW87}]
  \label{Theorem:Eps_apprx}
   There is a constant $c$ such that for any range space $\R=(X,\F)$ of finite VC-dimension $D$, any $0 < \epsilon\leq 1$, and any subset $A\subseteq X$, there exists an $\epsilon$-approximation for $(A,\F_A)$ whose cardinality is at most $c \cdot \frac{D}{\epsilon^2}\ln\frac{D}{\epsilon}$.
  
\end{theorem}

As was observed by Matou\v{s}ek, if $\R=(X,\F)$ is a range space with a bounded VC-dimension, then the fractional Helly number of the set family $\F$ must be bounded as well. 
Combined with the existence of small-size $\epsilon$-nets, this yields the following $(p,q)$-type theorem.

\begin{theorem}[$(p,q)$-Theorem for bounded VC-Dimension \cite{matousek2004bounded}]\label{Theorem:pq-boundedVC}
  \label{Theorem:boundedVC_pq} 
  Let $\R=(X,\F)$ be a range space whose dual shatter function satisfies $\pi_\R^\ast(m) = o(m^k)$ for some integer $k$, and let $p \geq k$. There exists a finite number $C(p,k,d)$ such that the following statement holds: any finite subfamily $\G \subseteq \F$ of non-empty sets that satisfies the $(p,k)$-property, admits a transversal of size at most $C$.
\end{theorem}

\noindent {\bf Semi-algebraic sets.} We recall that a set $A \subset \reals^d$ is {\it semialgebraic} if it can be defined by a Boolean combination of polynomial inequalities; that is, if 
	\[A = \{ x \in \reals^d : \Phi \left( f_1(x) \geq 0, f_2(x) \geq 0, \dots, f_s(x) \geq 0 \right) \}\] 
	where $\Phi$ is a Boolean formula and $f_1, \dots, f_s \in  \reals[x_1, \dots, x_d]$ are polynomials in the $d$ coordinates $x_1,\dots,x_d$ of $\reals^d$. We say that the description complexity of $A$ is bounded by an integer $c$ if it admits a semi-algebraic representation of the above form with $\max\{\deg(f_1),\dots,\deg(f_s),s\}\leq c$.

The definition of a semialgebraic set may also involve quantifiers. However, by a well-known result of Tarski, quantifiers can be eliminated; that is, each such set has an equivalent quantifier-free definition; See, e.g., \cite{bochnak2013real} for a comprehensive discussion of semialgebraic sets and quantifier elimination.

Let $\F_{d,c}$ denote the family of all the semialgebraic sets in $\reals^d$ whose respective description complexities are bounded by $c$.
Standard estimates on the number of sign patterns of real polynomials (due to Oleinik, Petrovskii, Milnor, Thom; see, e.g., \cite{pollacknumber} for precise results and references) imply that the shatter function of the range space $\R=(\reals^d,\F_{d,c})$ must satisfy $\pi_\R^\ast(m) \leq Cm^d$ for some constant $C = C(d,c)$. 
Plugging this observation into Theorem \ref{Theorem:boundedVC_pq} yields the following immediate corollary.

\begin{theorem}\label{Theorem:pq-semi}
For any positive integers $p$, $c$, and $d$ that satisfy $p\geq d+1$, there is a constant $C=C(p,c,d)$ with the following property:

 Any finite subfamily $\G \subseteq \F_{d,c}$ that satisfies the $(p,d+1)$-property admits a transversal of size at most $C$.
\end{theorem}

\section{Proof of Theorem \ref{Theorem:SplitHyperplanes}}\label{Sec:Hyperplanes}

\noindent {\bf Definition.} Let $P$ be a finite point set in $\reals^d$, and $0< \beta\leq \alpha \leq 1/2$. We say that a set $\tilde S\subseteq \reals^d$ is an {$(\alpha,\beta)$-core} of $P$ if for any $0\leq k\leq d-1$ the following conditions are met:

\begin{itemize}
	\item $\tilde S$ is intersected by any $\alpha$-splitting $k$-flat of $P$, and
	\item $P$ is $\beta$-split by any $k$-flat that intersects $\tilde S$.
\end{itemize}

\begin{lemma}
\label{Lemma:Core}

For any integer $d\geq 2$ and any fraction $0< \alpha \leq 1/2$, there is a fraction $\beta=\beta(\alpha,d)$ in $(0,\alpha)$ such that any point set $P$ in $\reals^d$ has an $(\alpha,\beta)$-core $\tilde{S}_\alpha(P)$ which is a compact and convex set.

\end{lemma}

\begin{figure}[h]
	\centering
	\includegraphics[scale = 0.8]{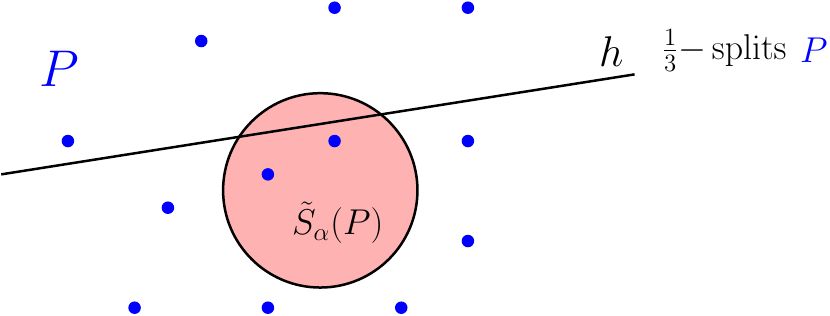}
	\caption{ \fontfamily{cmss} \footnotesize $(\alpha,\beta)$-core of $P$ with $\alpha = 1/3,\beta = 1/4$. Any line that $1/3$-splits $P$ intersects $\tilde{S}_\alpha(P)$ and any line that intersects $\tilde{S}_\alpha(P)$ at least $1/4$-splits $P$.}
	\label{fig:ab_core}
\end{figure}

Theorem \ref{Theorem:SplitHyperplanes} follows almost immediately by applying Theorem \ref{Theorem:AlonKalai} to the family $\{\tilde{S}_{\alpha}(P_i)\mid 1\leq i\leq n\}$ of $(\alpha,\beta)$-cores. We give a detailed proof of Theorem \ref{Theorem:SplitHyperplanes} at the end of this section.
	Therefore, most of this section is devoted to the proof of Lemma \ref{Lemma:Core}.

We will need the following result of Matou\v{s}ek \cite{Matouek1991EfficientPT}. 
\begin{theorem}[Matou\v{s}ek's Simplicial Partition \cite{Matouek1991EfficientPT}]
  \label{Theorem:partition}
  
In any dimension $d\geq 2$, one can fix a constant $b=b_d$ with the following property:

    For a finite point set $X \subset \mathbb{R}^d$ and any integer parameter $r$ that satisfies $2 \leq r < \abs{X}$, there is a partition $X = X_1 \mathbin{\dot\cup} \dotsb \mathbin{\dot\cup} X_t$ into at most $r$ parts such that $\abs{X}/r \leq \abs{X_i} \leq 2\abs{X}/r$ for all $i$, every set $X_i$ is enclosed by a simplex $\Delta_i$ and no hyperplane crosses more than $br^{1-1/d}$ of the simplices $\Delta_i$.
\end{theorem}

\begin{rem}
  In line with standard terminology, we say that a hyperplane $h$ \textbi{crosses} a simplex $\Delta$ if we have that $h\cap \Delta\neq \emptyset$ yet $\Delta\not\subset h$. If the set $P$ is not in general position (i.e., more than $k+1$ among its points lie in the same $k$-flat), it is possible that the dimension of some simplices in Theorem \ref{Theorem:partition} is smaller than $d$. 
\end{rem} 

\begin{proof}[Proof of Lemma \ref{Lemma:Core}.]
We prove the lemma with \[\beta(\alpha,d):= \frac{
\alpha}{2d+3}.\]

To construct the desired set $\tilde{S}_{\alpha}(P)$, let us fix an integer parameter 

\begin{equation}\label{Eq:ChooseR}
r:=\left\lceil \br{\frac{b_d(2d+2)(2d+3)}{\alpha}}^d\right\rceil.
\end{equation}

\noindent We may assume $d\geq 2$. Otherwise, let $p_1 \leq p_2 \leq \dots \leq p_{|P|}$ be the points of $P$, and let $x$ and $y$ be the points of $P$ such that there are exactly $\lceil \alpha|P| \rceil$ points of $P$ smaller than $x$ and exactly $\lceil \alpha|P| \rceil$ bigger than $y$. Setting $\tilde S_\alpha(P) = \conv{(\{x,y\})}$ concludes the proof in this case.\\

 \noindent We may also assume $|P| \geq r$, else we may take $\tilde S_\alpha(P)= \conv(P)$ as an $(\alpha,\beta)$-core. \\

\noindent Now, consider the simplicial partition $P = P_1 \dot\cup \cdots \dot\cup P_t$, for $t\leq r$, as described in Theorem \ref{Theorem:partition}. 
	
Let $T$ be an arbitrary \textit{transversal} to the set family $\{P_j\mid 1\leq j\leq t\}$, that is, a subset $T \subseteq P$ that satisfies $|T\cap P_j|=1$ for all $1\leq j\leq t$. To each subset of $S\in {T\choose s}$ of cardinality $s:=\ceil{\alpha r/2}$ we assign a single centerpoint $\varphi=\varphi(S)$ in accordance with Theorem \ref{Theorem:Centerpoint}.
Then our (tentative) $(\alpha,\beta)$-core $\tilde{S}_\alpha(P)$ is defined as the convex hull of all such centerpoints:
      \[\tilde{S}_\alpha(P)=\conv\left( \left\{\varphi(S): S\in {T\choose s}\right\} \right).\]

\begin{figure}[h]
	\centering
		\includegraphics[scale = 0.64]{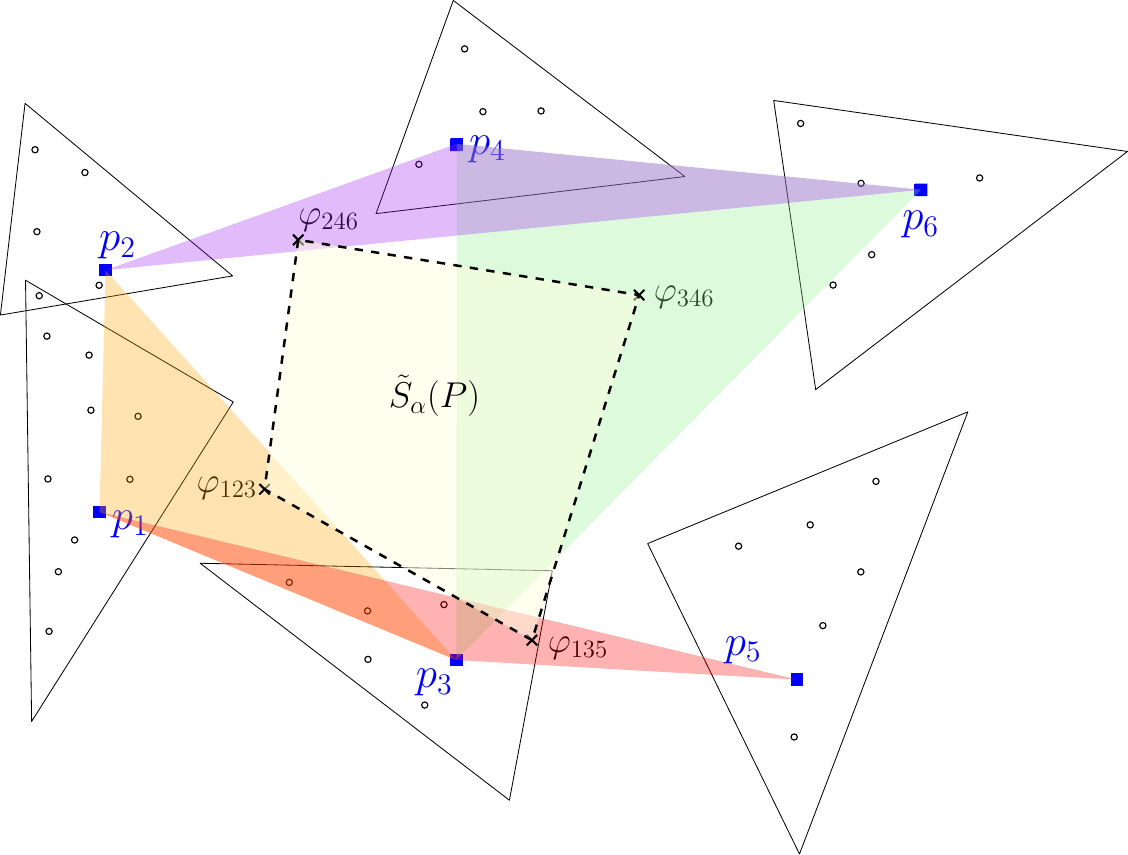}
		\caption{\fontfamily{cmss}\selectfont \footnotesize 
		Construction of an $(\alpha,\beta)-$core. For $t=6,r=12,\alpha = 1/3 \implies s=3$. The transversal $T$ is the set of blue points. For every $s$ points of $T$ we take a centerpoint (which exists by Rado's centerpoint theorem). Finally we take $\tilde{S}_\alpha(P)$ to be the convex hull of all the centerpoints $\varphi$.\\
		Note: For clarity sake, this illustration does not include \textbf{all} centerpoints of $s$-subsets of $T$.
		}
	\label{fig:core_construction}
\end{figure}

Observe that by Rado's centerpoint theorem, $\tilde{S}_\alpha(P)$ is not empty, and it is a compact convex set as the convex hull of a finite number of points.\\
To see that $\tilde{S}_\alpha(P)$ is crossed by any $k$-flat $\tau$ that $\alpha$-splits $P$, it is necessary and sufficient to check that it is crossed by any hyperplane $h$ through $\tau$. Assume for a contradiction that $h\cap \tilde{S}_\alpha(P)=\emptyset$. In other words, all the centerpoints $\varphi(S)$, for $S\in {T\choose s}$, lie in the same open half-space of $\reals^d\setminus h$, which we denote by $h^-$. On the other hand, since each closed half-space contains at least $\alpha |P|$ points of $P$, it follows that $\tilde{S}_\alpha(P)$ is strictly separated by $h$ from at least $\alpha|P|$ points which belong to the set $P^+=P\cap \overline{h^+}$ (where $\overline{h^+}$ denotes the closure of $h^+$). See Figure \ref{fig:many_hyperplane_crossings}.

Since each of the sets $P_j$ contains at most $2\abs{P}/r$ points, at least $s$ of them must intersect $P^+$. Let $p_1,\dots,p_{s}$ denote their respective transversal points within $T$, and fix $S:=\{p_1,\dots,p_s\}$. 
Since the centerpoint $q(S)$ lies in $h^-$, so do at least $s/ (d+1)$ of its elements $p_j$. For each of these latter indices $j$, the enclosing simplex $\Delta_j$ of $P_j$ contains a point from $h^-$ as well as a point from $\overline{h^+}$. Therefore, $h$ intersects at least $\ceil*{\alpha r/2} / (d+1)$ simplices in the partition. Together with the choice of $r$ in (\ref{Eq:ChooseR}), which guarantees that 

  \[
    br^{1-1/d}\leq  \frac{\alpha r}{2(d+1)},
    \]
    
\noindent this leads to a contradiction to the fact that no hyperplane crosses more than $br^{1-1/d}$ of the enclosing simplices $\Delta_j$.

\begin{figure}[h]
  \centering
  \includegraphics[scale = 0.65]{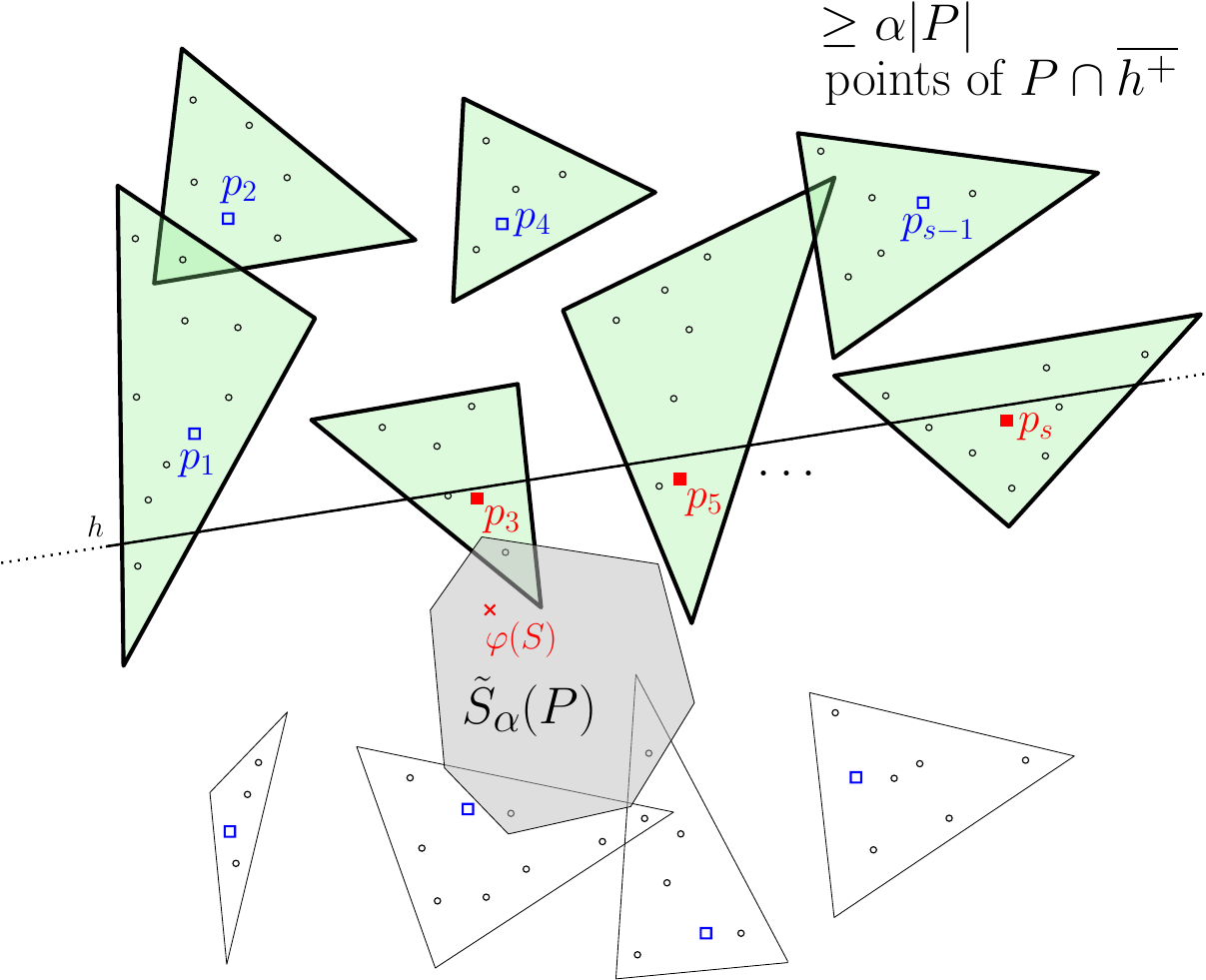}
 \caption{\fontfamily{cmss}\selectfont \footnotesize    
    Proof of Lemma \ref{Lemma:Core}. The empty circles are points of $P$, the boxes are the points of the transversal $T$, and the filled boxes are points among $\{p_1,\dotsc, p_s\}$ that are in $h^-$.
    If $h$ separates $\tilde{S}_\alpha(P)$ from at least $\alpha |P|$ points of $P$ (empty circles) then it must also separates the points of $S=\{p_1,\dotsc, p_s\}$ (red and blue boxes) from a centerpoint $\varphi(S)$ of $S$, which lies in $\tilde S_\alpha(P) \subset h^-$. Therefore, $h$ has to meet at least $s/(d+1) \geq \alpha r/2(d+1)$  of the enclosing simplices $\Delta_j$ whose respective points $p_j$ lie in $h^-$(simplices whose transversal points are red).}
  
  \label{fig:many_hyperplane_crossings}
\end{figure}

\newpage 
It remains to show that the set $P$ is $\beta$-split by any $k$-flat $\tau$ that intersects $\tilde{S}_\alpha(P)$. 
Indeed, fix any hyperplane $h$ through $\tau$, which must too cross $\tilde{S}_\alpha(P)$. In other words, each of the closed halfspaces $\overline{h^-}$ and $\overline{h^+}$ must contain at least one centerpoint $\varphi(S)$, for some $S\in {T\choose s}$, so both of them must contain at least $s/(d+1) = {\lceil{\alpha r/2}\rceil / (d+1)}\geq \frac{\alpha r}{2(d+1)}$ points of the transversal $T$, each belonging to a different part $P_j$. Since $h$ is an arbitrarly oriented hyperplane through $\tau$, it is sufficient to show that the closed halfspace $\overline{h^-}$ contains at least $\beta |P|$ points of $P$.

To recapitulate, the previous argument yields at least $\left\lceil{\frac{\alpha r}{2(d+1)}}\right\rceil$ simplices that intersect $\overline{h^-}$, of which at most $br^{1-1/d}$ can cross $h$, which leaves at least $\ceil{\frac{\alpha r}{2(d+1)}}-br^{1-1/d}$ simplices that are fully contained in $\overline{h^-}$. 
It, therefore, follows that 
\[
\left|P\cap \overline{h^-}\right| \geq  \left( \frac{\alpha r}{2(d+1)}-br^{1-1/d} \right)\frac{\abs{P}}{r} \geq \left( \frac{\alpha}{2(d+1)}-br^{-1/d} \right)\abs{P}.	
\]

\noindent In view of our choice (\ref{Eq:ChooseR}) of $r$, which guarantees that
\[
  \frac{b}{r^{1/d}}< \frac{\alpha}{2(d+1)}-\frac{\alpha}{2(d+1)+1},
  \] 

\noindent it follows that
\[
\left|P\cap \overline{h^-}\right|\geq \left(\frac{\alpha}{2(d+1)}-\br{ \frac{\alpha}{2(d+1)}-\frac{\alpha}{2(d+1)+1} }\right)|P| = \frac{\alpha}{2d+3}|P|.
\] 

\end{proof}

\begin{proof}[Proof of Theorem \ref{Theorem:SplitHyperplanes}.]

Let $\P=\set{P_1, \dotsc, P_n}$ be the point sets as in the statement of the theorem, and $\alpha$ a fraction in $(0,\frac{1}{2}]$. For each $1\leq i\leq n$, we construct the $(\alpha,\beta)$-core $\tilde{S_i}=\tilde{S}_\alpha(P_i)$ as described in Lemma \ref{Lemma:Core}, with $\beta(\alpha,d):=\frac{\alpha}{2d+3}$. Let $\S= \set{\tilde{S_1}, \dots, \tilde{S_n}}$ denote the resulting family of $n$ compact convex sets. For any subset $I\subseteq [n]=\{1,\dots,n\}$, we will use $\P_I$ and $\S_I$ to denote the respective subsets $\{P_i\mid i\in I\}$ and $\{\tilde{S}_i\mid i\in I\}$ of $\P$ and $\S$.

By the $(p,q,\alpha)$-property, any size-$p$ subset $I\in {[n]\choose p}$ must contain a size-$q$ subset $I'$ so that the family $\P_{I'}$ is $\alpha$-split by a hyperplane $h$. Clearly, that same hyperplane $h$ must also cross the respective set $\S_{I'}$ of $(\alpha,\beta)$-cores.
Now Theorem \ref{Theorem:AlonKalai} yields a set $Q$ of $C = C(p,q,d) < \infty$ hyperplanes with the property that every $(\alpha,\beta)$-core $\tilde{S}_i$ in $\S$ is intersected by some hyperplane $h\in Q$, which in turn guarantees that its respective point set $P_i\in \P$ is $\beta$-split by $h$.
\end{proof}

\section{Proof of Theorem \ref{Theorem:SplitFlats}}\label{Sec:Flats}

We identify every hyperplane in $\reals^d$ that does not pass through the origin, with a unique dual point $h^\ast=(a_1,\ldots,a_d)\in \reals^d\setminus \{0\}$, so that $h=\left\{(x_1,\ldots,x_d)\in \reals^d \mid a_1x_1+ \dotsb +a_dx_d=1\right\}$ \cite[Ch. 5.1]{matousek2013lectures}. 
At the center of our proof lies the following observation.

\begin{lemma}\label{Lemma:SemiHyperplanes}
For any positive integer $d$, and any positive numbers $\alpha\in (0,1/2]$ and $\epsilon\in (0,\alpha)$, one can fix a positive integer $c=c(\epsilon,d)<\infty$ with the following property.

For any finite point set $P$ in $\reals^d$, there exists a semi-algebraic set $\Gamma_{\alpha,\epsilon}(P)\subseteq \reals^d$, whose description complexity is bounded by $c$, so that the following statements hold.

\begin{enumerate}
	\item Any $\alpha$-splitting hyperplane $h$ for $P$ that does not pass through the origin yields a point $h^\ast\in \Gamma_{\alpha,\epsilon}(P)$.
	\item Any point $(a_1,\ldots,a_d)\in \Gamma_{\alpha,\epsilon}(P)$ describes a hyperplane that $(\alpha- \epsilon)$-splits $P$ (and does not pass through the origin).
\end{enumerate}
\end{lemma}

The proof of Lemma \ref{Lemma:SemiHyperplanes} will rely on the following straightforward corollary of Theorem \ref{Theorem:Eps_apprx}.

\begin{lemma}
\label{Lemma:ApprxSplit}
  Let $P\subset \reals^d$ be an $m$-point set, $1\leq k \leq d-1$, and let $\epsilon >0$. There exists a set $A\subseteq P$ of size at most $O\br{\frac{d}{\epsilon^2}\ln(\frac{d}{\epsilon})}$ with the following property:
  \begin{enumerate}[(i)]
    \item if $\tau$ is a $k$-flat that $\alpha$-splits $P$ then it $(\alpha-\epsilon)$-splits $A$;
    \item if $\tau$ is a $k$-flat that $\alpha$-splits $A$ then it $(\alpha-\epsilon)$-splits $P$.
  \end{enumerate}
  
\end{lemma}

\begin{proof}[Proof of Lemma \ref{Lemma:ApprxSplit}.]

Let us consider the hypergraph $(P,[\mathcal{H}_d]_P)$, where $\mathcal{H}_d$ denotes the family of all the open halfspaces in $\reals^d$, and $[\mathcal{H}_d]_P$ denotes the set family $\{P\cap h \mid h \in \mathcal{H}_d\}$. It is a well-known fact that VC-dim($P,\mathcal{H}_d$) = $d+1$, which is a direct consequence of Radon's lemma (see, e.g., \cite[Ch. 1.3]{matousek2013lectures} for a comprehensive exposition). Theorem \ref{Theorem:Eps_apprx} yields an $\epsilon$-approximation $A\subseteq P$ whose cardinality is at most $O\br{\frac{d+1}{\epsilon^2}\ln(\frac{d+1}{\epsilon})}=O\br{\frac{d}{\epsilon^2}\ln(\frac{d}{\epsilon})}$. 

To see part {\it (i)}, let $\tau$ be an $\alpha$-splitting $k$-flat for $P$, and $h$ be an arbitrary hyperplane through $\tau$. Let $h^+$ and $h^-$ denote the open halfspaces within $\reals^d\setminus h$, with labels assigned in an arbitrary manner. By the choice of $h$, we have that
                                \[\min\left\{\frac{\abs{P\cap h^+}}{\abs{P}}, \frac{\abs{P \cap h^-}}{\abs{P}}\right\} \geq \alpha.\]
                                
\noindent Since $A$ is an $\epsilon$-approximation, it follows that
                          \[\abs*{\frac{\abs{A\cap h^+}}{\abs{A}}-\frac{\abs{P\cap h^+}}{\abs{P}}} \leq \epsilon,\]
                          
\noindent which in turn means

        \[\frac{\abs{A\cap h^+}}{\abs{A}} \geq \frac{\abs{P\cap h^+}}{\abs{P}} - \epsilon \geq \alpha - \epsilon.\]

Since $h$ is an arbitrary hyperplane through $\tau$, and $h^+$ denotes an arbitrary halfspace of $\reals^d\setminus h$, it follows that $A$ is $(\alpha-\epsilon)$-split by $\tau$. The proof of part {\it (ii)} is symmetrical.
\end{proof}

\begin{proof}[Proof of Lemma \ref{Lemma:SemiHyperplanes}.]
We invoke Lemma \ref{Lemma:ApprxSplit} with parameter $\epsilon/2$. It can be assumed, with no loss of generality, that no point of $P$ coincides with the origin. 
For each hyperplane $h$ that does not pass through the origin $\vec{0}$, we use $h^-$ to denote the open half-space $\reals^d\setminus h$ that contains $\vec{0}$, and we use $h^+$ to denote the opposite open halfspace.

Our set $\Gamma_{\alpha,\epsilon}(P)$ is comprised of all the dual representations $h^\ast$ of the hyperplanes $h$ that $(\alpha-\epsilon/2)$-split the $\epsilon/2-$approximation $A$.
In view of the properties asserted in Lemma \ref{Lemma:ApprxSplit}, and of the fact that $|A|=O\left(\frac{d}{\epsilon^2}\ln \frac{d}{\epsilon}\right)$, it suffices to show that $\Gamma_{\alpha,\epsilon}(P)$ is a semi-algebraic set in $\reals^d$. 

To see the latter statement, note that $\Gamma_{\alpha,\epsilon}(P)$ is a union of finitely many sets of the form $X(P^-,P_0,P^+)$, where each set corresponds to a partition $P=P^- \,\dot{\cup}\,P_0 \,\dot{\cup}\, P^+$, and  ``represents'' all the hyperplanes $h$, excluding those that pass through the origin, that yield $P^-=P\cap h^-$, $P_0=P\cap h$, and $P^+=P\cap h^+$.
As is described, e.g., in a standard textbook \cite{bochnak2013real}, each of the latter sets must be semi-algebraic of the form $\Delta$ or $\Delta\setminus \{\vec{0}\}$, where $\Delta$ is a convex polyhedron, i.e., an intersection of a finite number of hyperplanes and/or open halfspaces. (As a matter of fact, the number of {\it non-empty} sets $X(P^-,P_0,P^+)$ is $O(|A|^d)$; furthermore, we are only interested in such partitions $(P^-,P_0,P^+)$ that attain $\max\{|P^+|,|P^-|\}\leq (1- \alpha+\epsilon/2)|P|$.)
\end{proof}

\noindent {\bf From hyperplanes to $\bm{k}$-flats.} 
To establish an analogue of Lemma \ref{Lemma:SemiHyperplanes} for $\alpha$-splitting $k$-flats, with $1\leq k\leq d-1$, we seek a comparably efficient dual representation of (almost all) $k$-flats by points in $\reals^{(k+1)(d-k)}$. Notice that while a generic $(k+1)$-tuple $B=\{b_1,\ldots,b_{k+1}\}$ of points in $\reals^d$ gives rise to a unique $k$-flat 

\begin{equation}
\tau(b_1, \dots, b_{k+1}) := \left\{ \sum_{i=1}^{k+1} t_i b_i \;\middle|\; 
t_1, \dots, t_{k+1} \in \mathbb{R},\; \sum_{i=1}^{k+1} t_i = 1 \right\}
\end{equation}

\begin{center}
	\text{\fontfamily{cmss} \footnotesize ($\tau$ is $k$-dimensional subspace which is the affine span of $b_1, \dots, b_{k+1}$)}
\end{center} 

\noindent which is the common intersection of all the hyperplanes through $b_1,\ldots,b_{k+1}$ \cite{Shafarevich2013}, this same flat $\tau$ is determined by {\it almost any} other choice of $k+1$ points $b_1,\ldots,b_{k+1}\in \tau$. 

To obtain a more ``economical'' representation, note first that $k$-flats in $\reals^d$ form a so-called Grassmanian manifold of dimension $(k+1)(d-k)$ in the projective space ${\mathbb P}^{d+1\choose k+1}$ \cite{Shafarevich2013}.
For a simpler (albeit incomplete) parametrization, let us fix a ``generic" $(k+1)$-tuple of $(d-k)$-flats $\lambda_1,\dots,\lambda_{k+1}$. To this end, it suffices to fix $k+1$ (ordered) $(d-k+1)$-point sets $B_1,\dots,B_{k+1}\subset \reals^d$, each chosen randomly and uniformly from $[0,1]^{d\times (d-k+1)}$, and set $\lambda_i=\tau(B_i)$ for all $1\leq i\leq k+1$.

To establish Theorem \ref{Theorem:SplitFlats}, it is necessary and sufficient to consider $q=(k+1)(d-k)+1$, and show that any family $\P=\{P_1,\dots,P_n\}$ of $n\geq p\geq q$ finite point sets in $\reals^d$ with the $(p,q)$-property can be $(\alpha-\epsilon)$-split by a fixed number of $k$-flats, whose number depends only on $d,k,p$ and $\epsilon$.
Notice that any given $k$-flat $\tau$ in $\reals^d$ admits, with probability $1$, a unique representation $\tau=\tau(b_1,\dots,b_{k+1})$, for some $(b_1,\ldots,b_{k+1})\in \lambda_1\times\dots\times \lambda_{k+1}\simeq \reals^{(k+1)(d-k)}$. (This is because almost all pairs of flats, with prescribed dimensions $k$ and $d-k$, determine a unique intersection point in $\reals^d$ \cite[Ch. 1.1]{matousek2013lectures};  see Figure \ref{fig:k_flats_representation}).

\begin{figure}[h]
  \centering
  \includegraphics[scale = 0.8]{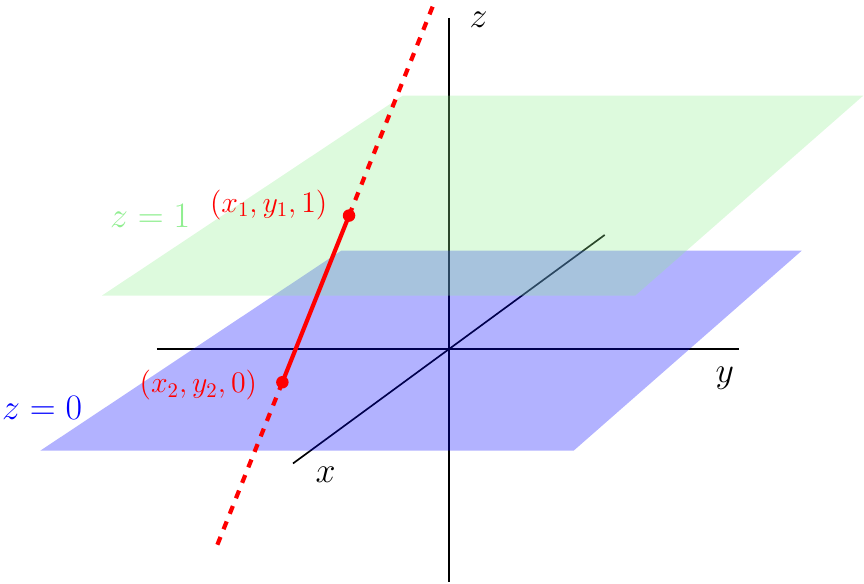}
 \caption{\fontfamily{cmss}\selectfont \footnotesize Illustration of the representation of $k$-flats in $\reals^3$ with $k=1$.
 We fix two 2-flats $z=0$ and $z=1$. Every point in each flat is represented by just two parameters and a generic $1$-flat is the affine span of two such points with probability 1. Thus, a generic $1$-flat is represented by 4 parameters ($=(d-k)(k+1)$).}
  \label{fig:k_flats_representation}
\end{figure}

It can, therefore, be assumed that every size-$q$ subset of $\P$, that is $\alpha$-split by at least one $k$-flat $\tau$, is $\alpha$-split but at least one $k$-flat of the latter kind. Lastly, a generic translation of the axes frame guarantees that none of these latter $k$-flats passes through the origin $\vec{0}$.

For each $1\leq i\leq n$ we consider the following semi-algebraic set in $\lambda_1\times\dots\times\lambda_{k+1}\simeq \reals^{(k+1)(d-k)}$

\begin{equation}\label{Eq:Lambda}
 \Lambda_{\alpha,\epsilon,k}(P_i):=
\end{equation}
 \[
 \left\{\overline{b}\in \oplus_{i=1}^{k+1}\lambda_i \mid \left[\vec{0}\not\in \tau(\overline{b})\right] \wedge\left[\forall \overline{a}\in \reals^d\setminus \{\vec{0}\}: \{b_1,\ldots,b_{k+1}\}\in (\overline{a})^\ast \Rightarrow (a_1,\ldots,a_d) \in \Gamma_{\alpha,\epsilon}(P_i)\right]\right\},
\]

\noindent where $\overline{b}$ denotes a $(k+1)$-tuple $(b_1,\ldots,b_{k+1})$ of points $b_i\in \lambda_i$, and $(\overline{a})^\ast$ denotes the hyperplane $\{(x_1,\ldots,x_d)\in \reals^d\mid a_1x_1+\ldots+a_dx_d=1\}$. Namely,  $\Lambda_{\alpha,\epsilon,k}(P_i)$ is comprised of all the ordered $(k+1)$-tuples $(b_1,\ldots,b_{k+1})\in \lambda_1\times\ldots\times \lambda_{k+1}$ with the following properties: (i) $\tau(b_1,\ldots,b_{k+1})$ does not contain $\vec{0}$, and (ii) any hyperplane through them (and, thus, through the entire flat $\tau(b_1,\ldots,b_{k+1})$), that misses the origin, belongs to the semi-algebraic set $\Gamma_{\alpha,\epsilon}(P_i)$ described in Lemma \ref{Lemma:SemiHyperplanes}.

\begin{lemma}
With the previous choice of $p\geq q=(k+1)(d-k)+1$, $\alpha$, $\epsilon$, $\P=\{P_1,\ldots,P_n\}$, and of the $(d-k)$-flats $\lambda_1,\ldots,\lambda_{k+1}$, the previously defined set family $\left\{\Lambda_{\alpha,\epsilon}(P_1),\ldots,\Lambda_{\alpha,\epsilon}(P_n)\right\}$ has the $(p,q)$-property. 
\end{lemma}

\begin{proof}
Fix an arbitrary $p$-size set $\P'=\{P_{i_1},\ldots,P_{i_p}\}\in {\P\choose p}$. According to the $(p,q,\alpha)$-property of $\P$, some $q$ sets $P_{i_j}$ of $\P'$ must be $\alpha$-split by a single $k$-flat $\tau$. 
Assume with no loss of generality that these sets have indices $i_1,\ldots,i_q$.
In the view of our choice of $\lambda_1,\ldots,\lambda_{k+1}$, it can be assumed that $\tau$ does not pass through the origin and, furthermore, admits a unique representation $\tau=\tau(b_1,\ldots,b_{k+1})$, for some $(b_1,\ldots,b_{k+1})\in \lambda_1\times\ldots\times \lambda_{k+1}$. Since the family $\{P_{i_1},\ldots,P_{i_{q}}\}$ is $\alpha$-split by any hyperplane $h$ through $\tau$, it follows, in particular, that every hyperplane through $\tau$ (that does not pass through $\vec{0}$) is represented by a dual point in $\bigcap_{j=1}^q\Gamma_{\alpha,\epsilon}(P_{i_j})$. Hence, we have that $(b_1,\ldots,b_{k+1})\in \bigcap_{j=1}^q\Lambda_{\alpha,\epsilon,k}(P_{i_j})$. The lemma now follows by repeating the previous argument for all choices of $\P'$.
\end{proof}

The family $\{\Lambda_{\alpha,\epsilon,k}(P_1),\ldots,\Lambda_{\alpha,\epsilon,k}(P_n)\}$ within $\lambda_1\times\ldots\times \lambda_{k+1}\simeq\reals^{(k+1)(d-k)}=\reals^{q-1}$ thus has a $(p,q)$-property and every member of it is a semi-algebraic set of bounded description complexity. According to Theorem \ref{Theorem:pq-semi}, this family admits a transversal\\ $\tilde{Q}\subseteq \lambda_1\times\ldots\times \lambda_{k+1}$ by at most $C$ points, where $C$ is some constant that depends only on the ambient dimension $q-1$, $p$, and the maximum description complexity of the semi-algebraic sets $\Lambda_{\alpha,\epsilon,k}(P_i)$, which in turn (and in view of Lemma \ref{Lemma:SemiHyperplanes}, and the definition (\ref{Eq:Lambda})) depends only on $d,k$, and $\epsilon$. 

To complete the proof of Theorem \ref{Theorem:SplitFlats}, let us fix $i\in \{1,\ldots,n\}$, and consider a $(k+1)$-tuple $(b_1,\ldots,b_{k+1})$ in $\Lambda_{\alpha,\epsilon,k}(P_i)\cap \tilde{Q}$. Since (i) $\tau(b_1,\ldots,b_{k+1})$ is an $l$-flat, for some $0\leq l\leq k$, that does not contain the origin $\vec{0}$, and (ii) any hyperplane through $\tau(b_1,\ldots,b_{k+1})$ that misses the origin, yields a point $\overline{a}\in \Gamma_{\alpha,\epsilon}(P_i)$, combining the definition of $\Gamma_{\alpha,\epsilon}(P_i)$ in Lemma \ref{Lemma:SemiHyperplanes} with a simple continuity argument implies that $P_i$ is $(\alpha+\epsilon)$-split by {\it any} hyperplane through $\tau(b_1,\ldots,b_{k+1})$ (including such hyperplanes that pass through $\vec{0}$). As a result, the set $P_i$ is $(\alpha-\epsilon)$-split by any $k$-flat through $b_1,\ldots,b_{k+1}$. Hence, the desired $(\alpha-\epsilon)$-splitting set for $\P$ can be obtained by fixing one $k$-flat through each $(k+1)$-tuple $(b_1,\ldots,b_k)\in \tilde{Q}$. $\qed$

\section{Concluding Remarks}\label{Sec:Conclude}
We have presented two different proofs of Theorem \ref{Theorem:SplitHyperplanes}, only one of which extends to flats of intermediate dimensionality $1\leq k\leq d-2$.
Since no $(p,q)$-theorems can exist for $k$-flat transversals to general families of convex sets, for $1\leq k\leq d-2$ \cite{alon2002transversal}, it seems unlikely that a general $(p,q)$-type statement similar to Theorem \ref{Theorem:SplitFlats} can hold with $\epsilon=0$.
However, it remains an intriguing question to determine if Theorem \ref{Theorem:SplitHyperplanes} can be established with $\beta=\alpha$. 
Another promising line of investigation is to determine whether the general ideas in Sections \ref{Sec:Hyperplanes} and \ref{Sec:Flats} can be used to make progress on additional quantitative or approximate Helly-type problems \cite{barany1982quantitative,de2015quantitative,Gao2008intristic_helly,ivanov2024helly}.

%bibliography
\printbibliography

\end{document}